\documentclass[10pt]{article}

\usepackage{times}
\usepackage{amsthm}
\usepackage{amsmath}
\usepackage{amssymb}
\usepackage{mathrsfs}

\def\A{{\mathcal A}}
\def\o{\omega}

\def\th{\theta}

\def\N{\mathbb{N}}
\def\T{\mathbb{T}}
\def\Z{\mathbb{Z}}

\def\B{{\mathcal B}}

\def\F{\mathcal F}
\def\H{\mathcal H}
\def\K{\mathscr K}
\def\M{{\mathcal M}}

\def\L{\mathcal L}
\def\U{\mathcal U}

\def\si{\sigma}
\def\Si{\Sigma}

\def\d{\mathrm{d}}

\def\Aut{\mathfrak{Aut}}

\def\G{\mathsf{G}}
\def\p{\parallel}

\def\<{\langle}
\def\>{\rangle}

\def\({\left(}
\def\){\right)}
\def\[{\left[}
\def\[{\right]}

\newtheorem{Theorem}{Theorem}[section]
\newtheorem{Remark}[Theorem]{Remark}
\newtheorem{Lemma}[Theorem]{Lemma}

\newtheorem{Corollary}[Theorem]{Corollary}
\newtheorem{Proposition}[Theorem]{Proposition}
\newtheorem{Definition}[Theorem]{Definition}
\newtheorem{Example}[Theorem]{Example}

\numberwithin{equation}{section}

\begin{document}


\title{Symmetry and Inverse Closedness for Some\\ Banach $^*$-Algebras Associated to Discrete Groups}

\date{\today}

\author{M. M\u antoiu 
}
\date{\small}
\maketitle \vspace{-1cm}


\begin{abstract}
A discrete group $\G$ is called rigidly symmetric if for every $C^*$-algebra $\A$ the projective tensor product $\ell^1(\G)\widehat\otimes\A$ is a symmetric Banach $^*$-algebra. For such a group we show that the twisted crossed product $\ell^1_{\alpha,\o}(\G;\A)$ is also a symmetric Banach $^*$-algebra, for every twisted action $(\alpha,\o)$ of $\G$ in a $C^*$-algebra $\A$\,. We extend this property to other types of decay, replacing the $\ell^1$-condition. We also make the connection with certain classes of twisted kernels, used in a theory of integral operators involving group $2$-cocycles. The algebra of these kernels is studied, both in intrinsic and in represented version.
\end{abstract}

\section{Introduction}\label{duci}

A Banach $^*$-algebra $\mathcal B$ is called {\rm symmetric} if the spectrum of $\,B^* B$ is positive for every $B\in\mathcal B$\,.
This happens (cf. \cite{Pa2}) if and only if the spectrum of any self-adjoint element is real. It is not known if the symmetry of $\mathcal B$ would imply the symmetry of the projective tensor product $\mathcal B\widehat\otimes\A$ for every $C^*$-algebra $\A$\,; such a property is called rigid symmetry.

All over this article $\G$ will be a discrete group with unit ${\sf e}$\,. If the convolution Banach $^*$-algebra $\ell^1(\G)$ is symmetric, $\G$ itself is called symmetric, while if it is rigidly symmetric, $\G$ will also be called rigidly symmetric. 

In \cite{Po} this terminology is applied to general locally compact groups. Various counterexamples are known. Certain solvable (thus amenable) groups and certain connected pollynomially growing groups are not symmetric \cite{LP}. Non-compact semi-simple Lie groups are never symmetric.

On the other hand, many interesting classes of groups are shown to be symmetric or even rigidly symmetric. If $\G$ is symmetric and ${\sf K}$ is compact, the semi-direct product ${\sf K}\rtimes\G$ is symmetric \cite{LP}. Compact \cite[Th. 1]{LP} or nilpotent locally compact groups \cite[Cor. 6]{Po} are rigidly symmetric. It is still not known if symmetry and rigid symmetry are equivalent for a locally compact group (see \cite{Po}). 

Besides the component-wise algebraic structure on $\ell^1(\G)\widehat\otimes\A\cong \ell^1(\G;\A)$\,, one can consider interesting and more complicated structures induced by twisted actions $(\alpha,\o)$ of $\G$ on the $C^*$-algebra $\A$\,, resulting in what is called twisted crossed products $\ell^1_{\alpha,\o}(\G;\A)$ \cite{Le,BS,PR1,PR2}. One recovers $\ell^1(\G;\A)$ if the action $\alpha$ is trivial ($\alpha_x$ is the identity map on $\A$) and the $2$-cocycle $\o$ is identically $1$\,. We recall the basic constructions in Section \ref{fourgo}, where we also replace the $\ell^1$-decay condition by others, described by admissible convolution algebras.

It is shown in \cite{FGL} that a discrete group $\G$ is rigidly symmetric if and only if $\ell^1_\alpha(\G;\A)$ is symmetric for every usual (untwisted) action $\alpha$ (see also  \cite{BB}). In Section \ref{tow} we extend this equivalence to include twisted $C^*$-dynamical systems, involving a $2$-cocycle $\o$\,, as well as other types of decay than those expressed by the $\ell^1$-condition. We also treat inverse closedness (also called spectral invariance, or Wiener property) of our twisted crossed products both inside enveloping $C^*$-algebras and faithfully represented as operators in Hilbert spaces. An extension, spectral invariance modulo a closed bi-sided $^*$-ideal, is exposed in \ref{trowak}, motivated by M. Lindner's work \cite{Li} on Fredholm operators. The results are illustrated in \ref{trow} by a series of Corollaries and Examples. An attempt to treat arbitrary locally compact groups failed, the obstacle being a non-trivial measurability issue; we are grateful to Professor Detlev Poguntke for pointing this out to us \cite{Pog}. 

In \cite[Sect. 4]{Ku} results are given about symmetry of $\ell^1$-twisted crossed products; the assumptions on the twisted $C^*$-dynamical system are quite strong. The nature of our result is different, showing for any discrete group that the presence of the twisted action is not relevant for symmetry issues. 
This would be particulary convenient if rigid symmetry will ever be shown to be equivalent to symmetry.

In Sections \ref{toww} and \ref{fourt} we describe the algebraic structure and the norm on a family of twisted kernels, connected to the twisted crossed product setting. They are useful in defining families of matrix operators whose composition involves cohomological factors. Actually, as they are defined, they form a class which is much larger than the one emerging from the theory of twisted crossed products, which can be recovered up to isomorphism only under a supplementary covariance condition (\ref{lilican}). We believe that the twisted matrix calculus has an interest of its own. When developed starting with the action of $\G$ on an Abelian $C^*$-algebra $\A$, it has interesting Hilbert-space representations that are studied in \ref{fuurt}. As a particular case, $\A$ could be a $C^*$-algebra of bounded complex functions on $\G$ (see \cite{BB,FGL}) and then a natural class of twisted kernels is already isomorphic to the twisted crossed product.  As a consequence of the results of the previous section, certain subfamilies of twisted kernels form symmetric Banach $^*$-algebra. Explicit faithful representations turn them into inverse closed subalgebras of bounded operators (in usual or in Fredholm sense) and all these are treated in \ref{firtoinog}; one gets an extension of the convolution dominated operators \cite{GKW}.

Besides the construction of the cohomological matrix calculus, most our results concern symmetry and inverse closedness, so they are part of what could be called noncommutative Wiener theory. The main purpose was to incorporate group $2$-cocycles in the presence of general types of decay. So we only cite references as \cite{Ba,Bas,BB,FS,FGL0,GKW,GL1,GL2,GL3,Kur} that have an immediate connection with our work. In particular, we follow rather closely some developments from \cite{FGL}. It is outside the scope of this paper to outline the history or the rich implications of Wiener's theory in its classical or its modern form. Most of the recent papers on this topic achieve this at least partly, while \cite{Gr,Kr} are excelent reviews exposing both the state of art of the subject and its numerous applications.

\section{Symmetric Banach $^*$-algebras associated to discrete groups}\label{fourtin}

\subsection{$\L$-type twisted crossed products}\label{fourgo}

Let us fix a discrete group $\G$\,. Recall that $\ell^1(\G)$ is a Banach $^*$-algebra with the usual $\ell^1$-norm, with the convolution product
\begin{equation}\label{libia}
(k\star l)(x):=\sum_{y\in\G}\,k(y)l(y^{-1}x)
\end{equation}
and with the involution
\begin{equation}\label{tunisia}
k^\star(x):=\overline{k(x^{-1})}\,.
\end{equation}

We also recall \cite{BeS} that a Banach space $\mathcal L(\G)$ of complex functions on $\G$ is called {\it a solid space of functions} if for any $k,l:\G\to\mathbb C$\,, if $|k(x)|\le|l(x)|$ everywhere and $l\in\mathcal L(\G)$ then $k\in\L(\G)$ and $\p\!k\!\p_{\L}\,\le\,\p\!l\!\p_{\L}$\,. Clearly $k$ and $|k|$ belong simultaneously to such a solid space of functions and their norms are the same.

\begin{Definition}\label{maroc}
We call {\rm admissible algebra} a subspace $\L(\G)$ of $\,\ell^1(\G)$ having its own norm $\p\!\cdot\!\p_{\L}$ which is stronger than the $\ell^1$-norm, such that
\begin{enumerate}
\item 
$\L(\G)$ is a solid space of functions,
\item
$\big(\L(\G),\star,^\star,\p\!\cdot\!\p_{\L}\big)$ is a unital Banach $^ *$-algebra. 
\end{enumerate}
\end{Definition}

Of course, $\ell^1(\G)$ itself is an admissible algebra. Another example is $\ell^{\infty,\vartheta}(\G):=\{k\mid \vartheta k\in \ell^\infty(\G)\}$\,, where $\vartheta:\G\to[1,\infty)$ is a subconvolutive weight, i.e. it satisfies $\vartheta^{-1}\star\vartheta^{-1}\le C\vartheta^{-1}$ for some constant $C$\,. Plenty of admissible algebras are contained in \cite{Fe,SW}. 

One can generate new admissible algebras by using weights.

\begin{Definition}\label{algeria}
A  {\rm submultiplicative, symmetric weight} is a function $\nu:\G\rightarrow[1,\infty)$ satisfying everywhere
\begin{equation}\label{slovenia}
\nu(xy)\le\nu(x)\nu(y)\,,\ \quad\nu(x^{-1})=\nu(x)\,.
\end{equation}
\end{Definition}

\begin{Lemma}\label{centrafricana}
Assume that $\L(\G)$ is an admissible algebra and $\nu$ is a weight on $\G$\,. Then 
\begin{equation}\label{niger}
\L^\nu(\G):=\{k\mid\nu k\in\L(\G)\}\,,\quad\ \p\!\cdot\!\p_{\L^\nu}\,:=\,\p\!\nu\,\cdot\!\p_{\L}
\end{equation} 
is an admissible algebra.
\end{Lemma}

\begin{proof}
Clearly $\p\!\cdot\!\p_{\L^\nu}$ is a complete norm, which is stronger than $\p\!\cdot\!\p_{\L}$ since $\nu(\cdot)\ge 1$ and $\L(\G)$ is solid. 
Since $\L(\G)$ is solid and $\nu$ is strictly positive, $\L^\nu(\G)$ is clearly also solid.

The weight $\nu$ being symmetric, one checks immediately that $(\nu k)^\star=\nu k^\star$ and this shows that $\L^\nu(\G)$ is stable under involution and that the involution is isometric.

By submultiplicativity of the weight one gets the inequality 
$$
|\nu(k\star l)|\le\nu(|k|\star|l|)\le|\nu k|\star|\nu l|\,.
$$ 
From it, using solidity, it follows that $\L^\nu(\G)$ is a subalgebra under convolution and that the norm $\p\!\cdot\!\p_{\L^\nu}$ is submultiplicative:
$$
\begin{aligned}
\p\!k\star l\!\p_{\L^\nu}\,&=\,\p\!\nu(k\star l)\!\p_{\L}\,\le\,\p\!|\nu k|\star|\nu l|\!\p_{\L}\\
&\le\,\p\!|\nu k|\!\p_{\L}\,\p\!|\nu l|\!\p_{\L}\,=\,\p\!\nu k\!\p_{\L}\,\p\!\nu l\!\p_{\L}\\
&=\,\p\!k\!\p_{\L^\nu}\p\!l\!\p_{\L^\nu}.
\end{aligned}
$$
\end{proof}

If $\A$ is a $C^*$-algebra, one denotes by $\Aut(\A)$ the group of its $^*$-automorphisms, by $\mathcal M(\A)$ its multiplier $C^*$-algebra (with the strict topology) and by $\mathcal U\mathcal M(\A)$ the corresponding unitary group.
Almost always $\A$ will be unital, so $\M(\A)$ will be identified with $\A$\,.

\begin{Definition}\label{olanda}
{\rm Twisted $C^*$-dynamical systems} $(\A,\alpha,\omega)$ are formed of a $C^*$-algebra $\A$\,, a map $\alpha:\G\rightarrow\Aut(\A)$ and a map $\o:\G\times\G\rightarrow\mathcal U\mathcal M(\A)$ satisfying for every $x,y,z\in\G$ 
\begin{equation}\label{luxemburg}
\alpha_x\circ\alpha_y={\sf ad}_{\o(x,y)}\circ\alpha_{xy}\,,
\end{equation} 
\begin{equation}\label{lichtenstein}
\o(x,y)\o(xy,z)=\alpha_x[\o(y,z)]\o(x,yz)\,,
\end{equation} 
\begin{equation}\label{flandra}
\omega(x,{\sf e})=1=\o({\sf e},x)\,.
\end{equation} 
\end{Definition}

If $(\A,\alpha,\o)$ is a twisted $C^*$-dynamical system, we denote by $\ell^{1}_{\alpha,\o}(\G;\A)$ the space $\ell^{1}(\G;\A)$ of integrable $\A$-valued functions on $\G$ endowed \cite{BS} with the composition law 
\begin{equation}\label{irlanda}
(f\diamond_{\alpha,\o}\!g)(x):=\sum_{y\in\G}\,f(y)\,\alpha_y\!\left[g(y^{-1}x)\right]\o(y,y^{-1}x)
\end{equation}
and the involution
\begin{equation}\label{islanda}
f^{\diamond_{\alpha,\o}}(x):=\o(x,x^{-1})^*\alpha_x\!\left[f(x^{-1})\right]^*.
\end{equation}
It is a Banach $^*$-algebra called {\it the $\ell^1$-twisted crossed product associated to} $(\A,\alpha,\o)$\,. 
The envelopping $C^*$-algebra of $\ell^{1}_{\alpha,\o}(\G;\A)$ \cite{PR1,PR2} is denoted by $\A\!\rtimes_\alpha^\o\!\G$ and called {\it the twisted crossed product $C^*$-algebra}.

\medskip
Now we are also given an admissible algebra $\L(\G)$\,. We define
\begin{equation}\label{tailanda}
\L(\G;\A):=\big\{f\in \ell^1(\G;\A)\mid\, \p\!f(\cdot)\!\p_\A\,\in\L(\G)\big\}
\end{equation}
 with norm $\p\!f\!\p_{\L(\G;\A)}\,:=\big\Vert\p\!f(\cdot)\!\p_\A\!\big\Vert_{\L}$\,. It is easy to see that $\L(\G;\A)$ is a $^*$-subalgebra of $\ell^1_{\alpha,\o}(\G;\A)$\,; when this structure is considered, we write $\L_{\alpha,\o}(\G;\A)$\,. This follows from the obvious estimations
$$
\p\!(f\diamond_{\alpha,\o}g)(x)\!\p_\A\,\le\,\sum_{y\in\G}\p\!f(y)\!\p_\A\,\p\!g(y^{-1}x)\!\p_\A\,=\,\big(\p\!f(\cdot)\!\p_\A\ast\p\!g(\cdot)\!\p_\A\big)(x)
$$
and
$$
\p\!f^{\diamond_{\alpha,\o}}(x)\!\p_\A\,=\,\p\!f(x^{-1})\!\p_\A\,=\,\p\!f(\cdot)\!\p_\A^\ast\!(x)
$$
and from the admissibility of $\L(\G)$\,. It is also easy to check that $\L_{\alpha,\o}(\G;\A)$ is a Banach $^*$-algebra with the (stronger) norm $\p\!f\!\p_{\L(\G;\A)}$\,. 

If $\alpha$ is the trivial action $\alpha_x(\varphi)=\varphi$ or if $\o=1$\,, they will desappear from the notation. The case $\omega=1$ leads to {\it the crossed product} \cite{Wi}. If in addition $\A=\mathbb C$ (with the trivial action) one recovers $\L(\G)$\,. We notice, for further use, that the  Banach space $\ell^1(\G;\A)$ can be identified \cite[1.10.11]{Pa1} with the projective tensor product $\ell^1(\G)\widehat\otimes\A$\,. For $\L(\G;\A)$\,, in general, there is no such a claim.

\begin{Definition}\label{ucise}
{\rm A covariant representation} $(\mathscr H,r,U)$ of the twisted $C^*$-dynamical system $(\A,\alpha,\o)$ is composed \cite{BS,PR1} of a Hilbert space $\mathscr H$, a non-degenerate $^*$-representation $r:\A\to\mathbb B(\mathscr H)$ and a unitary-valued map $U:\G\rightarrow\mathcal U(\mathscr H)$ satisfying for every $x,y\in\G$ and $\varphi\in\A$
\begin{equation}\label{cipru}
U(x)U(y)=r[\o(x,y)]U(xy)\ \quad{\rm and}\ \quad U(x)r(\varphi)U(x)^*=r[\alpha_x(\varphi)]\,.
\end{equation}
\end{Definition}

Given a covariant representation $\big(\mathscr H,r,U\big)$\,, {\rm the integrated form} \cite[pag. 512]{BS}
\begin{equation}\label{dobrogea}
(r\rtimes U)(f):=\sum_{x\in\G}r[f(x)]U(x)
\end{equation}
provides a $^*$-representation $r\rtimes U:\ell^1(\G;\A)\rightarrow\mathbb B(\mathscr H)$ that extends to the twisted crossed product $C^*$-algebra $\A\!\rtimes_\alpha^\o\!\G$ and restricts to any of the Banach $^*$-algebras $\L_{\alpha,\o}(\G;\A)$\,. The extension and the restrictions are all contractive.

\begin{Remark}\label{ulei}
{\rm It is known that covariant representations $\big(\mathscr H,r,U\big)$ with faithful $r$ exist. This will be used in Section \ref{tow}, so we indicate a construction.

Let $\pi:\A\to\mathbb B(\H)$ be a $^*$-representation of the $C^*$-algebra $\A$ in a separable Hilbert space $\H$\,. We can inflate $\pi$ to a $^*$-representation of $\A$ in $\mathscr H:=\ell^2(\G;\H)\cong \ell^2(\G)\otimes\H$ by
\begin{equation}\label{fleasca}
[r^\pi(\varphi)v](x):=\pi\big[\alpha_{x^{-1}}(\varphi)\big]v(x)\,.
\end{equation}
It is obvious that $r^\pi$ is injective if $\pi$ is injective. One also defines for every $y\in\G$
\begin{equation}\label{pleasca}
\big[L^\pi_\o(y) v\big](x):=\pi\big[\omega(x^{-1}\!,y)\big]v(y^{-1}x)\,.
\end{equation}
It is straightforward to show that $\big(\mathscr H,r^\pi,L^\pi_\o\big)$ is a covariant representation; we say that {\it it is induced by $\pi$}\,. 
A related version, involving right translations, can be found in \cite[pag. 517]{BS} and \cite[Def. 3.10]{PR1} for instance.
}
\end{Remark}

\subsection{Symmetry and inverse closedness of $\L$-type twisted crossed products}\label{tow}

\begin{Definition}\label{anglia}
Let $\L(\G)$ be an admissible algebra over the discrete group $\G$\,. 
Then $\G$ is called {\rm TCP-rigidly $\L$-symmetric} (rigidly symmetric in the sense of twisted crossed products for $\L$-type decay) if the Banach $^*$-algebra $\L_{\alpha,\o}(\G;\A)$ is symmetric for every twisted $C^*$-dynamical system $(\A,\alpha,\o)$ with group $\G$\,. 
\end{Definition}

If this is required only for the trivial case $(\alpha,\o)=({\sf id},1)$\,, we speak of {\it rigid $\L$-symmetry}. We also drop $\L$ if $\L(\G)=\ell^1(\G)$\,, to reach standard terminology \cite{LP,Po}\,.

\begin{Definition}\label{derrin}
The $^*$-subalgebra $\mathfrak B$ of the unital $C^*$-algebra $\,\mathfrak C$ is called {\rm an inverse closed} (or {\rm spectral}, or {\rm Wiener}) {\rm subalgebra} if for every $f\in\mathfrak B$ that is invertible in $\mathfrak C$ one has $f^{-1}\in\mathfrak B$\,.
\end{Definition}

\medskip
The following result, extending \cite[Cor.\,1]{FGL}, shows that for every admissible algebra $\L(\G)$ the discrete group $\G$ is TCP-rigidly $\L$-symmetric if (and only if) it is rigidly $\L$-symmetric.

\begin{Theorem}\label{monaco}
Let $(\A,\alpha,\o)$ be a twisted $C^*$-dynamical system with rigidly $\L$-symmetric discrete group $\G$\,.
\begin{enumerate}
\item
$\L_{\alpha,\o}(\G;\A)$ is a symmetric Banach $^*$-algebra.
\item
$\L_{\alpha,\o}(\G;\A)$ is an inverse closed subalgebra of its enveloping $C^*$-algebra. In particular, if $\,\G$ is rigidly symmetric, $\ell^1_{\alpha,\o}(\G;\A)$ is spectral in $\A\!\rtimes_\alpha^\o\!\G$\,.
\item
Let $\,\Pi:\A\!\rtimes_\alpha^\o\!\G\rightarrow\mathbb B(\mathscr H)$ be a faithful $^*$-representation; then $\Pi\big[\ell^1_{\alpha,\o}(\G;\A)\big]$ is inverse-closed in $\mathbb B(\mathscr H)$\,.
\end{enumerate}
\end{Theorem}

The next basic Lemma is inspired by \cite[Prop. 2]{FGL}, to which it reduces if $\o=1$\,, $\A=\ell^\infty(\G)$ and $\L(\G)=\ell^1(\G)$\,; see also \cite[Prop. 2.7]{BB}.

\begin{Lemma}\label{danemarca}
Let $(\A,\alpha,\o)$ be a twisted $C^*$-dynamical system with discrete group $\G$ and $\L(\G)$ an admissible algebra. There exists a $C^*$-algebra $\mathscr B$ and an isometric $^*$-morphism 
$$
\th:\L_{\alpha,\o}(\G;\A)\rightarrow \L_{{\sf id},1}(\G;\mathscr B)\equiv \L(\G;\mathscr B)\,.
$$
\end{Lemma}

\begin{proof}
We use a covariant representation $(\mathscr H,r,U)$ (cf. Definition \ref{ucise}) of the twisted $C^*$-dynamical system $(\A,\alpha,\o)$ with $r$ faithful and choose $\mathscr B\,$ to be a $C^*$-subalgebra of $\,\mathbb B(\mathscr H)$ containing $r(\A)U(\G)$. Then we set 
\begin{equation}\label{kamchatka}
\th:\L_{\alpha,\o}(\G;\A)\rightarrow \L(\G;\mathscr B)\,,\quad(\th f)(x):=r[f(x)]U(x)\,.
\end{equation}

Clearly $\th$ is well-defined and isometric:
$$
\begin{aligned}
\p\! \th f\!\p_{\L(\G;\mathscr B)}\,&=\big\Vert\!\p\!(\th f)(\cdot)\!\p_{\mathscr B}\!\big\Vert_{\L(\G)}\\
&=\big\Vert\!\p\!r[f(\cdot)]U(\cdot)\!\p_{\mathscr B}\!\Vert_{\L(\G)}\\
&=\big\Vert\!\p\!r[f(\cdot)]\!\p_{\mathscr B}\!\big\Vert_{\L(\G)}\\
&=\,\p\!f\!\p_{\L(\G;\A)}\,,
\end{aligned}
$$
since $U(x)$ is unitary and $r$, being faithful, is isometric.

For two elements $f,g$ of $\L_{\alpha,\o}(\G;\A)$ one computes using (\ref{cipru}) and the definitions
$$
\begin{aligned}
(\th f\star\th g)(x)&=\sum_{y\in\G}(\th f)(y)(\th g)(y^{-1}x)\\
&=\sum_{y\in\G}r[f(y)]\,U(y)\,r\!\left[g(y^{-1}x)\right]U(y^{-1}x)\\
&=\sum_{y\in\G}r[f(y)]\,U(y)\,r\!\left[g(y^{-1}x)\right]U(y)^* U(y)U(y^{-1}x)\\
&=\sum_{y\in\G}r[f(y)]\,r\!\left[\alpha_y\!\left(g(y^{-1}x)\right)\right] r[\o(y,y^{-1}x)]\,U(x)\\
&=r\!\left[(f\diamond_{\alpha,\o}\!g)(x)\right]U(x)\\
&=\left[\th\!\left(f\diamond_{\alpha,\o}\!g\right)\right]\!(x)\,.
\end{aligned}
$$
Finally we treat the involution, using the identity $U(x^{-1})=U(x)^*\,r\!\left[\o(x,x^{-1})\right]$\,:
$$
\begin{aligned}
(\th f)^\star(x)&=(\th f)(x^{-1})^*\\
&=\left(r\!\left[f(x^{-1})\right]U(x^{-1})\right)^*\\
&=U(x^{-1})^*\,r\!\left[f(x^{-1})\right]^*\\
&=r\!\left[\o(x,x^{-1})^*\right]U(x)\,r\!\left[f(x^{-1})\right]^*U(x)^*U(x)\\
&=r\!\left[\o(x,x^{-1})^*\right]\,r\!\left(\alpha_x\!\left[f(x^{-1})\right]^*\right)U(x)\\
&=r\!\left[\th\!\left(f^{\diamond_{\alpha,\o}}\right)\right]U(x)\\
&=\left[\th\!\left(f^{\diamond_{\alpha,\o}}\right)\right]\!(x)\,.
\end{aligned}
$$
\end{proof}

\begin{Remark}\label{galicia}
{\rm Using the precise notation $\th=\th_{r,U}$\,, the integrated form (\ref{dobrogea}) can be written as
$r\!\rtimes\!U=I\circ\th_{r,U}$\,, in terms of the $^*$-morphism
$I:\ell^{1}(\G;\mathscr B)\rightarrow\mathscr B$ given by $I(F):=\sum_{x\in\G}F(x)$\,.
}
\end{Remark}

We are now in a position to {\it prove Theorem \ref{monaco}}.

\begin{proof}
1. It is known \cite[Th.11.4.2]{Pa2} that symmetry of a Banach $^*$-algebra is inherited by its closed $^*$-algebras. This, Lemma \ref{danemarca} and the fact that $\L(\G;\mathscr B)$ was assumed symmetric prove the result.

\medskip
2. We recall \cite{Pa2} that a $^*$-algebra is called {\it reduced} if its universal $C^*$-seminorm is in fact a norm. It is known \cite[11.4]{Pa2} that a reduced Banach $^*$-algebra is symmetric if and only if it is a spectral subalgebra of its enveloping $C^*$-algebra. Thus, by point 1, we only need to know that $\L_{\alpha,\o}(\G;\A)$ is reduced. But a $^*$-subalgebra of a reduced $^*$-algebra is also reduced \cite[Prop. 9.7.4]{Pa2}. By Lemma \ref{danemarca} $\,\L_{\alpha,\o}(\G;\A)$ is isometrically isomorphic to a $^*$-subalgebra of $\L(\G;\mathscr B)\subset \ell^1(\G;\mathscr B)$\,, so everything follows from the fact that $\ell^1(\G;\mathscr B)$ is reduced.

\medskip
3. Follows immediately from 2, from obvious properties of isomorphisms and from the fact that any $C^*$-algebra is inverse closed in a larger $C^*$-algebra.
\end{proof}

\subsection{Inverse closedness modulo ideals}\label{trowak}

In a $C^*$-algebra we will call briefly {\it ideal} a closed self-adjoint bi-sided ideal.

\begin{Definition}\label{dorin}
Let $\mathfrak J$ be an ideal  of the unital $C^*$-algebra $\mathfrak C$\,. The $^*$-subalgebra $\mathfrak B$ of $\,\mathfrak C$ is called {\rm $\mathfrak J$-inverse closed} if for every $f\in\mathfrak B$ such that there are elements $g\in\mathfrak C$\,, $h,k\in\mathfrak J$ with $fg=1_\mathfrak C+h$ and $gf=1_\mathfrak C+k$ one actually has $g\in\mathfrak B$\,.
\end{Definition}

One can rephrase: {\it $\mathfrak B$ is $\mathfrak J$-inverse closed in $\mathfrak C$ if and only if $\,\mathfrak B/\mathfrak J$ (shorthand for $\,\mathfrak B/(\mathfrak B\cap\mathfrak J)$) is inverse closed in $\mathfrak C/\mathfrak J$\,.} If $\mathfrak J=\{0\}$ the notion coincides with that introduced in Definition \ref{derrin}. 

\medskip
Suppose now that $(\A,\alpha,\o)$ is a twisted action of the discrete group $\G$ and that $\mathcal J$ is an ideal of $\A$ that is $\alpha$-invariant: $\alpha_x(\mathcal J)\subset\mathcal J$ for every $x\in\G$\,. We denote by the same letter $\alpha$ the action of $\G$ by automorphisms of $\mathcal J$ defined by restrictions. On the other hand, the unital $C^*$-algebra $\A$ is naturally embedded in the multiplier algebra $\M(\mathcal J)$ \cite{Wi}, so $\o(x,y)$ can be seen as a multiplier of $\mathcal J$ for every $x,y\in\G$\,. Finally one gets the $C^*$-dynamical system $(\mathcal J,\alpha,\o,\G)$\,.  It is known \cite{PR2} that the twisted crossed product $\mathcal J\!\rtimes_\alpha^\o\!\G$ may be identified with an ideal of $\mathcal A\!\rtimes_\alpha^\o\!\G$\,. Under this identification, $\ell^1_{\alpha,\o}(\G;\mathcal J)$ becomes an ideal of $\ell^1_{\alpha,\o}(\G;\A)$ in the natural way: the $\ell^1$-function $f:\G\rightarrow \mathcal J$ is taken to be $\A$-valued. Now we use the exactness of the twisted crossed product construction to prove

\begin{Theorem}\label{dorina}
Assume that the discrete group $\G$ is rigidly symmetric. Then the Banach $^*$-algebra $\ell^1_{\alpha,\o}(\G;\A)$ is $\mathcal J\!\rtimes_\alpha^\o\!\G$-inverse closed in the twisted crossed product $\mathcal A\!\rtimes_\alpha^\o\!\G$ for every $\alpha$-invariant ideal $\mathcal J$ of $\A$\,.
\end{Theorem}

\begin{proof}
Setting $\,\mathfrak C:=\mathcal A\!\rtimes_\alpha^\o\!\G$\,,\, $\mathfrak J:=\mathcal J\!\rtimes_\alpha^\o\!\G$ and $\mathfrak B:=\ell^1_{\alpha,\o}(\G;\A)$\,, we must show that $\mathfrak B/\mathfrak J$ is inverse closed in $\mathfrak C/\mathfrak J$\,. Note that $\mathfrak B\cap\mathfrak J=\ell^1(\G;\mathcal J)$\,. One has a natural quotient twisted $C^*$-dynamical system $\big(\A/\mathcal J,\tilde\alpha,\tilde\o,\G\big)$ given by
\begin{equation}\label{valentin}
\tilde\alpha_x(\varphi+\mathcal J):=\alpha_x(\varphi)+\mathcal J\quad{\rm and}\quad\tilde\o(x,y):=\o(x,y)+\mathcal J\,.
\end{equation}
Then the quotient $\mathfrak B/\mathfrak J\equiv\mathfrak B/(\mathfrak B\cap\mathfrak J)=\ell^1_{\alpha,\o}\big(\G;\A\big)/\ell^1_{\alpha,\o}\big(\G;\mathcal J\big)$ is isomorphic to $\ell^1_{\tilde\alpha,\tilde\o}\big(\G;\A/\mathcal J\big)$\,, which is a symmetric Banach $^*$-algebra, by our Theorem \ref{monaco} and the fact that $\G$ was assumed rigidly symmetric. Thus it is inverse closed in its enveloping $C^*$-algebra, that can be identified \cite{PR2} to the quotient $\mathfrak C/\mathfrak J$\,.
\end{proof}

Our main motivation for introducing Definition \ref{dorin} and proving Theorem \ref{dorina} comes from the article \cite{Li}. 
We recall that a bounded operator $\,T$ in a Hilbert space $\mathscr H$ is called {\it Fredholm} if it has a closed range, a finite-dimensional kernel and its adjoint $T^*$ also has a finite-dimensional kernel. Let us denote by $q:\mathbb B(\mathscr H)\to\mathbb B(\mathscr H)/\mathbb K(\mathscr H)$ the canonical surjection. 
By Atkinson's Theorem, $T$ is Fredholm if and only if its canonical image $q(T)$ in the Calkin algebra $\mathbb B(\mathscr H)/\mathbb K(\mathscr H)$ is invertible. In other terms, there should exist $S\in\mathbb B(\mathscr H)$ and $K,L\in\mathbb K(\mathscr H)$ such that 
\begin{equation}\label{ruanda}
ST=1+K\,,\quad\ TS=1+L\,.
\end{equation}
One would like to know if the information upon $S$ can be automatically improved.

\begin{Definition}\label{burundi}
Let $\mathfrak F$ be a $^*$-algebra of bounded operators in $\mathscr H$ containing $\mathbb K(\mathscr H)$\,. We say that it is {\rm Fredholm inverse closed} if for every Fredholm element $T\in\mathfrak F$ there exist $S\in\mathfrak F$ and $K,L\in\mathbb K(\mathscr H)$ such that (\ref{ruanda}) holds. 
\end{Definition}

Clearly $\mathfrak F$ is Fredholm inverse closed if and only if $q(\mathfrak F)$ is a spectral $^*$-subalgebra of the Calkin algebra and this fits Definition \ref{dorin}. 

Let $(\A,\alpha,\o,\G)$ be a twisted $C^*$-dynamical system with discrete rigidly symmetric group $\G$ and $\mathcal J$ an $\alpha$-invariant (closed, self-adjoint bi-sided) ideal in $\A$\,. Let $\Pi:\mathcal A\!\rtimes_\alpha^\o\!\G\rightarrow\mathbb B(\mathscr H)$ be a faithful $^*$-representation such that $\Pi\big[\mathcal J\!\rtimes_\alpha^\o\!\G\big]=\mathbb K(\mathscr H)$ (using consacrated terminology, $\mathcal J\!\rtimes_\alpha^\o\!\G$ is an {\it elementary $C^*$-algebra}). Applying Theorem \ref{dorina} one gets immediately

\begin{Corollary}\label{dorinel}
The Banach $^*$-algebra $\Pi\big[\ell^1_{\alpha,\o}(\G;\A)\big]$ is Fredholm inverse-closed.
\end{Corollary}

In Example \ref{dorinet} and Corollary \ref{papita} we are going to present concrete versions of this result. In \cite{Li} the group $\G$ is $\Z^n$ and there is no cohomological factor $\o$\,. On the other hand many of the refinements of \cite{Li} are not available by the methods of the present article.

\subsection{Consequences and examples}\label{trow}

We start with an abstract consequence of Theorem \ref{monaco}.

\begin{Corollary}\label{chile}
The quotient of a (TCP-)rigidly symmetric discrete group by a normal subgroup is TCP-rigidly symmetric.
\end{Corollary}

\begin{proof}
Suppose that ${\sf N}$ is a closed normal subgroup of $\G$\,. Describing $\ell^1(\G/{\sf N})$ is not trivial, but it has been done in \cite[Th. pag. 146]{Pa1}. Without giving all the details, let us just say that a surjective $^*$-morphism $\Phi:\ell^1(\G)\to \ell^1(\G/{\sf N})$ exists, which defines an isometric $^*$-isomorphism $\ell^1(\G/{\sf N})\cong \ell^1(\G)/{\sf ker}(\Phi)$\,.
Thus $\ell^1(\G/{\sf N};\B)\cong \ell^1(\G/{\sf N})\widehat\otimes\B\cong \big[\ell^1(\G)/{\sf ker}(\Phi)\big]\widehat\otimes\B$ for every $C^*$-algebra $\B$\,. Using \cite[Prop 1.10.10]{Pa1} we see that $\big[\ell^1(\G)/{\sf ker}(\Phi)\big]\widehat\otimes\B$ can be identified to the quotient $\big[\ell^1(\G)\widehat\otimes\B\big]\,/\,{\sf ker}\big(\Phi\widehat\otimes{\sf id}_\B\big)$\,. To conclude, one gets the isometric isomorphism of Banach $^*$-algebras 
$$
\ell^1(\G/{\sf N};\B)\cong \ell^1(\G;\B)\,/\,{\sf ker}\big(\Phi\widehat\otimes{\sf id}_\B\big)\,.
$$ 
It is also known \cite[Th.11.4.2]{Pa2} that the quotient of a symmetric Banach $^*$-algebra by a closed bi-sided $^*$-ideal is symmetric. Hence $\G/{\sf N}$ is rigidly symmetric. Combining this with Theorem \ref{monaco} finishes the proof.
\end{proof}

\begin{Corollary}\label{sevastopol}
\begin{enumerate}
\item
 Finite extensions of discrete nilpotent groups are TCP-rigidly symmetric. In particular \cite{Gr}, discrete finitely-generated groups of polynomial growth are TCP-rigidly symmetric.
\item
If $\,{\sf Z}$ is a central subgroup with $\G/{\sf Z}$ rigidly symmetric, then $\G$ is TCP-rigidly symmetric. 
\end{enumerate}
\end{Corollary}

\begin{proof}
1. For the first assertion, we use once again Theorem \ref{monaco} and invoke \cite[Cor. 3]{LP} for rigid symmetry of finite extensions of discrete nilpotent groups.

\medskip
2. The second assertion is a consequence of Theorem \ref{monaco} and \cite[Th. 7]{LP}.
\end{proof}

The next result follows directly from Lemma \ref{danemarca}, so it relies only on a symmetry assumption.

\begin{Proposition}\label{podolia}
If the discrete group $\,\G$ is symmetric and the twisted $C^*$-dynamical system $(\A,\alpha,\o)$ admits a covariant representation $(\mathscr H,r,U)$ with $r$ faithful and $r(\A)U(\G)$ contained in a type I $\,C^*$-algebra $\mathscr B\subset\mathbb B(\mathscr H)$\,, then $\ell^1_{\alpha,\o}(\G;\A)$ is a symmetric Banach $^*$-algebra.
\end{Proposition}

\begin{proof}
The projective tensor product of a symmetric Banach $^*$-algebra and a type I $\,C^*$-algebra is a symmetric Banach $^*$-algebra \cite[Th. 1]{Ku}. Thus $\ell^1(\G;\mathscr B)\cong \ell^1(\G)\widehat\otimes\mathscr B$ is symmetric if $\G$ is (only) symmetric and $\mathscr B$ is type I and then Lemma \ref{danemarca} finishes the proof.
\end{proof}

It is not easy to exploit this result in an explicit non-trivial way. If $\G$ is amenable and the twisted crossed product $\A\!\rtimes_\alpha^\o\!\G$ happens to be type I, it can be used in Proposition \ref{podolia}. But criteria for such a property are difficult to give even if $\o$ is trivial; we refer to \cite[7.5]{Wi} for a discusssion.

To illustrate this Proposition with simple but non-trivial examples, let us take $\A=\mathbb C$ and (thus) $\alpha_x={\sf id}_\mathbb C$ for every $x\in\G$\,. Then the $2$-cocycle (in this case also called {\it multiplier}) $\o$ is $\mathbb T$-valued and $\ell^1_{{\sf id},\o}(\G;\mathbb C)=:\ell^1_\o(\G)$ is the $\o$-twisted $\ell^1$-algebra of the  group $\G$\,. The isometric $^*$-morphism $\th$ defined in (\ref{kamchatka}) reads now 
$$
\th:\ell^1_\o(\G)\rightarrow \ell^1(\G;\mathscr B)\,,\quad(\th f)(x):=f(x)U(x)\,,
$$
where $U:\G\to\U(\mathscr H)$ is an $\o$-projective representation, i.e. it satisfies
$$
U(x)U(y)=\o(x,y)U(xy)\,,\quad\forall\,x,y\in\G
$$
and we choose $\mathscr B$ (say) to be the $C^*$-subalgebra of $\mathbb B(\mathscr H)$ generated by $U(\G)$\,. Using now Proposition \ref{podolia} one gets

\begin{Corollary}\label{coreea}
If the discrete group $\G$ is symmetric and admits an $\o$-projective representation $U$ generating a type I $\,C^*$-algebra, then $\ell^1_\o(\G)$ is symmetric. In particular, this happens when $\G$ is amenable and symmetric and the twisted group $C^*$-algebra $C^*_\o(\G)$ is type I.
\end{Corollary}

Conditions for a discrete group to have at least one type I $\o$-representation are in \cite[Th. 1]{Kl}, to which we send the interested reader; see also \cite{H}. 

\medskip
Lemma \ref{centrafricana} tells us that the Beurling algebra $\ell^{1,\nu}(\G)$ is an admissible algebra if $\nu$ is a submultiplicative symmetric weight. 

\begin{Corollary}\label{japita}
Let $\G$ be a rigidly symmetric amenable discrete group and $\nu$ a sumultiplicative symmetric weight. 
Assume that there exists a generating subset $V$ of $\,\G$ containing the unit ${\sf e}$ such that 
\begin{enumerate}
\item
the following uGRS (uniform Gelfand-Raikov-Shilov) condition holds:
\begin{equation}\label{ugrs}
\lim_{n\rightarrow\infty}\sup_{x_1,\dots,x_n\in V}\nu(x_1\cdots x_n)^{1/n}=1\,,
\end{equation} 
\item
for some finite constant $C$ one has for any  $n\in\N$ 
\begin{equation}\label{magaoaie}
\sup_{x\in V^n\setminus V^{n-1}}\nu(x)\le C\!\inf_{x\in V^n\setminus V^{n-1}}\nu(x)\,.
\end{equation}
\end{enumerate}
Then $\ell^{1,\nu}_{\alpha,\o}(\G;\A)$ is a symmetric Banach $^*$-algebra for every twisted $C^*$-dynamical system $(\A,\alpha,\o)$\,.
\end{Corollary}

\begin{proof}
From Theorem \ref{monaco} we know that $\G$ is TCP-rigidly $\ell^{1,\nu}$-symmetric whenever it is rigidly $\ell^{1,\nu}$-symmetric. 
The problem of the symmetry of $\ell^{1,\nu}(\G;\mathscr B)$ for a discrete group and an arbitrary $C^*$-algebra $\mathscr B$ has been discussed and solved in \cite[Sect. 5]{FGL}, relying on the assumptions 1 and 2. 
\end{proof}

\begin{Example}\label{nct}
Non-commutative tori \cite{Bo,GL3} {\rm are obtained setting $\,\G:=\mathbb Z^n$\,, $\L:=\l^1$ and $\A:=\mathbb C$\,. Thus the two-cocycle $\o:\mathbb Z^n\times\mathbb Z^n\rightarrow\mathbb T$ is a multiplier, the action $\alpha$ must be trivial, the twisted crossed product is the $\ell^1$-twisted group algebra $\ell^1_\o(\mathbb Z^n)$ with enveloping $C^*$-algebra $C^*_\o(\mathbb Z^n)$\,. Up to cohomology, the multipliers of $\Z^n$ are given by skew-symmetric matrices $(\th_{x,y})_{x,y\in\Z^n}$ through 
\begin{equation}\label{be}
\o_\th(x,y)=\exp\big(2\pi i\th_{x,y}\big)\,,\quad\forall\,x,y\in\Z^n\,.
\end{equation} 
If $\,\o=1$ one deals with the ($\l^1$- and $C^*$-) group algebras of $\,\mathbb Z^n$, which are commutative; one has $C^*(\mathbb Z^n)\cong C(\mathbb T^n)$\, by using the Fourier transform.

Unconventionally, for every admissible (convolution) algebra $\L(\mathbb Z^n)$\,, the Banach $^*$-algebra $\L_\o(\mathbb Z^n)$ may be called {\it the $\o$-noncommutative torus of $\L$-decay}. Our results imply that it is symmetric and inverse-closed in its enveloping $C^*$-algebra (and in its faithful representations) if $\mathbb Z^n$ is rigidly $\L$-symmetric. This holds if $\L=\ell^1$ since $\Z^n$ is Abelian. It also holds for $\L=\ell^{1,\nu}$ if $\nu$ is a GRS-weight.

Such results (and others) have been first obtained in \cite{GL1,GL2,GL3}. As outlined in \cite{GL3}, the symmetry of $\ell^1_\o(\mathbb Z^n)$ can be proved by embedding it isometrically into the non-commutative convolution algebra $L^1(\G_\o)$\,, where $G_\o$ is the central group extension of the $1$-dimensional torus $\T$ (the unitary group of $\mathbb C$) by $\Z^n$ associated to the multiplier $\o$\,. Such a strategy is impossible in more complicated situations; note for example that the unitary group of an infinite-dimensional $C^*$-algebra $\A$ is not locally compact, so it does not posses a Haar measure.

One gets a faithful representation of $C^*_\o(\mathbb Z^n)$ by starting as in Remark \ref{ulei} with the one-dimensional representation ${\rm i}$ of $\A=\mathbb C$ in the Hilbert space $\mathbb C$\,.
One gets the integrated form
\begin{equation}\label{guagua}
{\rm i}\!\rtimes L^{\rm i}_\o:\ell^1_\o(\mathbb Z^n)\rightarrow\mathbb B\big[\ell^2(\Z^n)\big]\,,\quad (i\!\rtimes L^{\rm i}_\o)f=\sum_{x\in\mathbb Z^n}\!f(x)L^{\rm i}_\o(x)
\end{equation}
which has the form of a twisted convolution
\begin{equation}\label{guaguita}
\big([({\sf i}\!\rtimes L^{\sf i}_\o)f]v\big)(z)=\sum_{y\in\Z^n}\!\o(-z,z-y)f(z-y)v(y)\,.
\end{equation}
One can state an inverse-closedness result in terms of this representation.


}
\end{Example}

\begin{Example}\label{dorinet}
{\rm We illustrate now Fredholm inverse closedness. One starts with an amenable discrete rigidly symmetric group $\G$ and a closed unital $^*$-subalgebra $\A(\G)$ of $\ell^\infty(\G)$ containing the ideal 
\begin{equation}\label{drorin}
c_0(\G):=\big\{\varphi:\G\rightarrow\mathbb C\,\big\vert\,\varphi(x)\underset{x\rightarrow\infty}{\longrightarrow}0\big\}
\end{equation} 
and stable under translations: if $\varphi\in\A(\G)$ and $x\in\G$ then $\big[\alpha_x(\varphi)\big](\cdot):=\varphi(x^{-1}\cdot)\in\A(\G)$\,. Let also $\,\o:\G\times\G\rightarrow\A(\G)$ be a $2$-cocyle with respect to $\alpha$\,. The formula 
\begin{equation}\label{dorinache}
[\Pi(f)u](x):=\sum_{y\in\G}f\big(xy^{-1};x\big)\o\big(x^{-1},xy^{-1};{\sf e}\big)u(y)
\end{equation}
defines a $^*$-representations $\Pi:\ell^1_{\alpha,\o}(\G;\A(\G))\rightarrow\mathbb B\big[\ell^2(\G)\big]$ that extends to a faithful $^*$-representation of $\A(\G)\!\rtimes_\alpha^\o\!\G$ (see Section \ref{firtoinog}). It is known that $\Pi\big[c_0(\G)\!\rtimes_\alpha^\o\!\G\big]=\mathbb K\big[\ell^2(\G)\big]$\,. Hence we are in the framework of Section \ref{trowak} and, by Corollary \ref{dorinel}, $\,\Pi\big[\ell^1(\G;\A(\G))\big]$ is a Fredholm inverse closed Banach $^*$-algebra of operators in $\ell^2(\G)$\,. This can be applied to $\A(\G)=\ell^\infty(\G)$\,. In Corollary \ref{papita} we will give an interpretation in terms of twisted matrix operators.
}
\end{Example}

\section{The twisted kernel calculus}\label{furtin}

\subsection{Algebras of twisted kernels}\label{toww}

As before, we are given a twisted $C^*$-dynamical system $(\A,\alpha,\o)$ with discrete group $\G$\,. 
When defining algebras of kernels, for simplicity, {\it we are going to assume that the $2$-cocycle $\o$ is center-valued}; consequently (\ref{luxemburg}) will read simply $\alpha_x\circ\alpha_y=\alpha_{xy}$ for every $x,y\in\G$\,. The general case can be treated, but some formulae are more complicated. The $2$-cocycle identity (\ref{lichtenstein}) will be needed below in the form
\begin{equation}\label{ucraina}
\alpha_{s^{-1}}\!\big[\o(m,n)\big]\alpha_{s^{-1}}\!\big[\o(mn,r)\big]=\alpha_{s^{-1}m}\!\big[\o(n,r)\big]\alpha_{s^{-1}}\!\big[\o(m,nr)\big]\,.
\end{equation}

For $\A$-valued kernels on $\G$\,, i.e. functions $K:\G\times\G\rightarrow\A$\,, one defines (formally) the composition
\begin{equation}\label{suediaa}
(K\bullet_{\alpha,\o}\!L)(x,y):=\sum_{z\in\G}K(x,z)L(z,y)\,\alpha_{x^{-1}}\!\big[\o(xz^{-1},zy^{-1})]
\end{equation}
and the involution
\begin{equation}\label{norvegiaa}
K^{\bullet_{\alpha,\o}}(x,y):=\alpha_{x^{-1}}\!\big[\o(xy^{-1},yx^{-1})^*\big]\,K(y,x)^*\,.
\end{equation}

\begin{Definition}\label{uzbechistan}
Let $\L(\G)$ be an admissible algebra. Let us denote by $\mathscr K_{\alpha,\o}^\L(\G\times\G;\A)$ the set of all functions $K:\G\times\G\rightarrow\A\,$ for which $\p\!K\!\p_{\K^\L}$ is finite; we set
\begin{equation}\label{ingusia}
\p\!K\!\p_{\K^\L}:=\inf\left\{\p\!k\!\p_{\L}\,\mid\, \p\! K(x,y)\!\p_\A\,\le\,|k(xy^{-1})|\,,\ \,\forall\,x,y\in\G\right\}.
\end{equation}
The elements of $\K_{\alpha,\o}^\L(\G\times\G;\A)$ will be called {\rm $\A$-valued convolution-dominated kernels (or matrices) of type $\L$}\,. 
\end{Definition}

The space $\L(\G)$ describes the type of off-diagonal decay possessed by the elements of $\mathscr K_{\alpha,\o}^\L(\G\times\G;\A)$\,. If $\L(\G)=\ell^{1,\nu}(\G)$ we prefer the notation $\K_{\alpha,\o}^\nu(\G\times\G;\A)$ and if $\L(\G)=\ell^{1}(\G)$ we skip the upper index.
The lower index $(\alpha,\o)$ is only justified by the fact that  on $\K_{\alpha,\o}^\L(\G\times\G;\A)$ we are going to consider the algebraic structure defined by (\ref{suediaa}) and (\ref{norvegiaa}). The starting point in defining Banach spaces through norms of the form (\ref{ingusia}) seems to be \cite{GKW}, in which the group $\G$ is $\mathbb Z^n$\,, one has $\L=\ell^1$ and the $2$-cocycle is absent.

\begin{Remark}\label{eritreea}
{\rm Given an element $K$ of the space $\K_{\alpha,\o}^\L(\G\times\G;\A)$\,, let us introduce the notation $\kappa(x):=\sup_{z\in\G}\!\p\!K(z,x^{-1}z)\!\p_\A$\,. It is easy to show that $\kappa\in\L(\G)$ and that $\p\!K\!\p_{\K^\L}\,=\,\p\!\kappa\!\p_{\L}$\,. 
}
\end{Remark}

\begin{Proposition}\label{osetia}
$\left(\K_{\alpha,\o}^\L(\G\times\G;\A),\bullet_{\alpha,\o},^{\bullet_{\alpha,\o}},\p\!\cdot\!\p_{\K^\L}\right)$ is a Banach $^*$-algebra.
\end{Proposition}

\begin{proof}
Straightforwardly, if $\p\!K(x,y)\!\p_\A\,\le|k(xy^{-1})|$ and $\p\!L(x,y)\!\p_\A\,\le|l(xy^{-1})|$ everywhere, then 
$$
\p\!(K\bullet_{\alpha,\o} L)(x,y)\!\p_\A\,\le(|k|\star|l|)(xy^{-1})\,,\ \forall\,x,y\in\G\,.
$$ 
This shows immediately that $\mathscr K_{\alpha,\o}^\L(\G\times\G;\A)$ is stable under $\bullet_{\alpha,\o}$ and that 
$$
\p\!K\bullet_{\alpha,\o} L\!\p_{\K^\L}\,\le\,\p\!K\!\p_{\K^\L}\p\!L\!\p_{\K^\L}\,,\quad\ \forall\,K,L\in\mathscr K_{\alpha,\o}^\L(\G\times\G;\A)\,. 
$$
The map $^{\bullet_{\alpha,\o}}$ is a well-defined isometry, since 
$$
\p\!K(y,x)\!\p_\A\,\le |k(xy^{-1})|\ \Longleftrightarrow\ \p\!K^{\bullet_{\alpha,\o}}(x,y)\!\p_\A\,\le|k^*(xy^{-1})|
$$
for all $x,y\in\G$ and the involution $^*$ is $\p\!\cdot\!\p_{\L}$-isometric.

\medskip
\noindent
To show associativity, we rely on the identity
$$
\alpha_{x^{-1}}\!\big[\o(xb^{-1},ba^{-1})\big]\,\alpha_{x^{-1}}\!\big[\o(xa^{-1},ay^{-1})\big]=\alpha_{b^{-1}}\!\big[\o(ba^{-1},ay^{-1})\big]\,\alpha_{x^{-1}}\!\big[\o(xb^{-1},by^{-1})\big]\,,
$$
which follows from (\ref{ucraina}) setting $m=xb^{-1}$\,, $n=ba^{-1}$, $r=ay^{-1}$ and $s=x$\,. 

\medskip
\noindent
The map $^{\bullet_{\alpha,\o}}$ is involutive because of the identity
$$
\alpha_{x^{-1}}\!\big[\o(xy^{-1},yx^{-1})\big]=\alpha_{y^{-1}}\!\big[\o(yx^{-1},xy^{-1})\big],
$$
which follows by setting in (\ref{ucraina}) $m=xy^{-1}$, $n=yx^{-1}$, $r=xy^{-1}$, $s=x$ and then using (\ref{flandra}).

\medskip
The identity $(K\bullet_{\alpha,\o}\!L)^{\bullet_{\alpha,\o}}=L^{\bullet_{\alpha,\o}}\!\bullet_{\alpha,\o}\!K^{\bullet_{\alpha,\o}}$ is equivalent with
\begin{equation}\label{egiptt}
\begin{aligned}
\alpha_{x^{-1}}\!\big[\o(xy^{-1},yx^{-1})\big]&\,\alpha_{y^{-1}}\!\big[\o(yz^{-1},zx^{-1})\big]\,\alpha_{x^{-1}}\!\big[\o(xz^{-1},zy^{-1})\big]=\\
&\alpha_{x^{-1}}\!\big[\o(xz^{-1},zx^{-1})\big]\,\alpha_{z^{-1}}\!\big[\o(zy^{-1},yz^{-1})\big]\,.
\end{aligned}
\end{equation}
To prove (\ref{egiptt}) first notice that, by straightforward particularizations in (\ref{ucraina}), one can write
$$
\alpha_{x^{-1}}\!\big[\o(xz^{-1},zy^{-1})\big]\,\alpha_{x^{-1}}\!\big[\o(xy^{-1},yx^{-1})\big]=\alpha_{z^{-1}}\!\big[\o(zy^{-1},yx^{-1})\big]\,\alpha_{x^{-1}}\!\big[\o(xz^{-1},zx^{-1})\big]
$$
and this reduces (\ref{egiptt}) to 
$$
\alpha_{y^{-1}}\!\big[\o(yz^{-1},zx^{-1})\big]\,\alpha_{z^{-1}}\!\big[\o(zy^{-1},yx^{-1})\big]=\alpha_{z^{-1}}\!\big[\o(zy^{-1},yz^{-1})\big]\,.
$$
This one holds, taking in (\ref{ucraina}) $m=zy^{-1}$\,, $n=yz^{-1}$\,, $r=zx^{-1}$ and $s=z$ and applying (\ref{flandra}).
\end{proof}

Clearly, each $\mathscr K_{\alpha,\o}^\L(\G\times\G;\A)$ can be seen as a $^*$-subalgebra of the Banach $^*$-algebra $\mathscr K_{\alpha,\o}(\G\times\G;\A)$\,, the one corresponding to the maximal choice $\L(\G):=\ell^1(\G)$\,. The correspondence $\L(\G)\mapsto\mathscr K_{\alpha,\o}^\L(\G\times\G;\A)$ is increasing (obvious meaning: smaller spaces also have stronger norms). One has relations as 
$$
\mathscr K_{\alpha,\o}^{\L_1}(\G\times\G;\A)\bullet_{\alpha,\o}\mathscr K_{\alpha,\o}^{\L_2}(\G\times\G;\A)\subset\mathscr K_{\alpha,\o}^{\L_1\ast\L_2}(\G\times\G;\A)\,;
$$
thus if $\L_1(\G)$ is an ideal of $\L_2(\G)$ under convolution, then $\mathscr K_{\alpha,\o}^{\L_1}(\G\times\G;\A)$ is an ideal of $\mathscr K_{\alpha,\o}^{\L_2}(\G\times\G;\A)$ under the composition $\bullet_{\alpha,\o}$\,.

\medskip
We connect now the Banach $^*$-algebras $\mathscr K_{\alpha,\o}^{\L}(\G\times\G;\A)$ with the twisted crossed products.

\begin{Proposition}\label{finlanda}
Let us define $(\Gamma f)(x,y):=\alpha_{x^{-1}}\!\big[f(xy^{-1})\big]$\,. 
\begin{enumerate}
\item
Then $\Gamma:\L_{\alpha,\o}(\G;\A)\rightarrow\mathscr K^\L_{\alpha,\o}(\G\times\G;\A)$ is an isometric $^*$-morphism.
\item
The range of $\,\Gamma$ is the space of covariant kernels
\begin{equation}\label{lilican}
\begin{aligned}
\mathscr K_{\alpha,\o}^{\L}&(\G\times\G;\A)_{\rm cov}:=\\
&\big\{K\in\mathscr K_{\alpha,\o}^{\L}(\G\times\G;\A)\,\mid\,K(xz,yz)=\alpha_{z^{-1}}[K(x,y)]\,,\ \forall\,x,y,z\in\G\big\}
\end{aligned}
\end{equation}
and the inverse $\Gamma^{-1}$ reads on $\mathscr K_{\alpha,\o}^{\L}(\G\times\G;\A)_{\rm cov}$
\begin{equation}
\big(\Gamma^{-1}K\big)(x):=\alpha_x\big[K(x,{\sf e})\big].
\end{equation}
\end{enumerate}
\end{Proposition}

\begin{proof}
1. The product: one has
$$
\begin{aligned}
\left[\Gamma(f\diamond_{\alpha,\o}\!g)\right]\!(x,y)&=\alpha_{x^{-1}}\!\big[(f\diamond_{\alpha,\o}\!g)(xy^{-1})\big]\\
&=\sum_{a\in\G}\alpha_{x^{-1}}\!\big[f(a)\big]\,\alpha_{x^{-1}a}\!\big[g(a^{-1}xy^{-1})\big]\,\alpha_{x^{-1}}\!\big[\o(a,a^{-1}xy^{-1})\big]\\
&=\sum_{z\in\G}\alpha_{x^{-1}}\!\big[f(xz^{-1})\big]\,\alpha_{z^{-1}}\!\big[g(zy^{-1})\big]\,\alpha_{x^{-1}}\!\big[\o(xz^{-1},zy^{-1})\big]\\
&=\sum_{z\in\G}(\Gamma f)(x,z)\,(\Gamma g)(z,y)\,\alpha_{x^{-1}}\!\big[\o(xz^{-1},zy^{-1})\big]\\
&=\left[(\Gamma f)\bullet_{\alpha,\o}\!(\Gamma g)\right]\!(x,y)\,.
\end{aligned}
$$
The involution: one has
$$
\begin{aligned}
(\Gamma f)^{\bullet_{\alpha,\o}}(x,y)&=\alpha_{x^{-1}}\!\big[\o(xy^{-1},yx^{-1})^*\big]\,(\Gamma f)(y,x)^*\\
&=\alpha_{x^{-1}}\!\big[\o(xy^{-1},yx^{-1})^*\big]\,\alpha_{y^{-1}}\!\big[f(yx^{-1})^*\big]\\
&=\alpha_{x^{-1}}\!\big[f^{\diamond_{\alpha,\o}}(xy^{-1})\big]=\left[\Gamma(f^{\diamond_{\alpha,\o}})\right]\!(x,y)\,.
\end{aligned}
$$
To prove that $\Gamma$ is isometric one writes
$$
\begin{aligned}
\p\!\Gamma(f)\!\p_{\mathcal K^\L}&=\inf\left\{\p\!k\!\p_{\L(\G)}\,\mid\, \p\! \alpha_{x^{-1}}\!\big[f(xy^{-1})\big]\!\p_\A\,\le\,|k(xy^{-1})|\,,\ \,\forall\,x,y\in\G\right\}\\
&=\inf\left\{\p\!k\!\p_{\L(\G)}\,\mid\, \p\!| f(z)\!\p_\A\,\le\,|k(z)|\,,\ \,\forall\,x,z\in\G\right\}\\
&=\big\Vert\p\!f(\cdot)\!\p_\A\big\Vert_{\L(\G)}=\,\p\!f\!\p_{\L\!(\G;\A)}.
\end{aligned}
$$

2. is quite straightforward.
\end{proof}

\begin{Corollary}\label{terminal}
The Banach $^*$-algebra $\mathscr K_{\alpha,\o}^{\L}(\G\times\G;\A)_{\rm cov}$ is symmetric.
\end{Corollary}

\begin{proof}
This follows from Theorem \ref{monaco} and Proposition \ref{finlanda}.
\end{proof}

\subsection{Algebras of twisted operators}\label{fourt}

The main motivation for introducing the Banach $^*$-algebra $\mathscr K_{\alpha,\o}^{\L}(\G\times\G;\A)$ consists in its connection with a twisted version of matrix operators. 

Let $\pi:\A\to\mathbb B(\H)$ be a $^*$-representation of the $C^*$-algebra $\A$ in a Hilbert space $\H$\,. One represents $\mathscr K_{\alpha,\o}^{\L}(\G\times\G;\A)$ by twisted matrix operators in $\ell^2(\G;\H)$ through
\begin{equation}\label{ghana}
\big[{\sf Int}^\pi_\o (K)v\big](x):=\sum_{y\in\G}\pi\big[K(x,y)\big]\pi\big[\o(x^{-1},xy^{-1})\big]v(y)\,.
\end{equation}
The label "twisted" indicates the presence of the cohomological factor $\pi\big[\o(x^{-1},xy^{-1})\big]$\,. If $\o=1$ one recovers a natural notion of matrix operator with vector-valued kernel $\pi\circ K:\G\times\G\to\mathbb B(\H)$\,.

\begin{Proposition}\label{elvetia}
${\sf Int}^\pi_\o:\K_{\alpha,\o}^\L(\G\times\G;\A)\rightarrow\mathbb B\!\left[\ell^2(\G;\H)\right]$ is a contractive $^*$-representation. 
\end{Proposition}

\begin{proof}
Of course, it is enough to treat the case $\L(\G)=\ell^1(\G)$\,. In addition, it is known that any $^*$-representation of a Banach $^*$-algebra is contractive.

Simple computations show that ${\sf Int}^\pi_\o(K)\,{\sf Int}^\pi_\o(L)={\sf Int}^\pi_\o(K\bullet_{\alpha,\o}\!L)$ and ${\sf Int}^\pi_\o(K)^*={\sf Int}^\pi_\o\!\left(K^{\bullet_{\alpha,\o}}\right)$\,. Let us sketch them. For the product one has
$$
\begin{aligned}
\big[{\sf Int}_\lambda(K)&\,{\sf Int}_\lambda(L)u\big](x)\\
&=\sum_{y\in\G}\sum_{z\in\G}\pi\big[K(x,z)\big]\pi\big[L(z,y)\big]\pi\big[\o(x^{-1},xz^{-1})\big]\pi\big[\o(z^{-1},zy^{-1})\big]u(y)\\
&=\sum_{y\in\G}\pi\!\left\{\sum_{z\in\G}K(x,z)L(z,y)\alpha_{x^{-1}}\!\big[\o(xz^{-1},zy^{-1})\big]\right\}\pi\big[\o(x^{-1};xy^{-1})\big]u(y)\\
&=\sum_{y\in\G}\pi\big[(K\bullet_{\alpha,\o}\!L)(x,y)\big]\pi\big[\o(x^{-1};xy^{-1})\big]u(y)\\
&=\left[{\sf Int}^\pi_\o(K\bullet_{\alpha,\o}\!L)u\right]\!(x)\,.
\end{aligned}
$$
The form of the adjoint is computed bellow:
$$
\begin{aligned}
\big\<{\sf Int}^\pi_\o&(K)u,v\big\>_{\ell^2(\G;\H)}=\sum_{y\in\G}\left\<\left[{\sf Int}^\pi_\o(K)u\right]\!(y),v(y)\right\>_{\H}\\
&=\sum_{y\in\G}\left\<\sum_{x\in\G}\pi\big[K(y,x)\big]\pi\big[\o(y^{-1}\!,yx^{-1})\big]u(x),v(y)\right\>_{\!\H}\\
&=\sum_{y\in\G}\sum_{x\in\G}\left\<u(x),\pi\big[\o(y^{-1}\!,yx^{-1})\big]^*\pi\big[K(y,x)\big]^*v(y)\right\>_{\!\H}\\
&=\sum_{x\in\G}\left\<u(x),\sum_{y\in\G}\pi\left\{\alpha_{x^{-1}}\!\big[\o(xy^{-1}\!,yx^{-1})^*\big]K(y,x)^*\right\}\pi\big[\o(x^{-1}\!,xy^{-1})\big]v(y)\right\>_{\!\H}\\
&=\sum_{x\in\G}\big\<u(x),\big[{\sf Int}^\pi_\o(K^{\bullet_{\alpha,\o}})v\big](x)\big\>_\H\\
&=\big\<u,{\sf Int}^\pi_\o(K^{\bullet_{\alpha,\o}})v\big\>_{\ell^2(\G;\H)}.
\end{aligned}
$$
For the forth equality we used the $2$-cocycle identity and the fact that $\o$ is central-valued.
\end{proof}

For the same $\pi$\,, the integrated form ${\sf ind}^\pi_\o:=r^\pi\rtimes L^\pi_\o$ described in (\ref{dobrogea}) and corresponding to the covariant representation given in (\ref{fleasca}) and (\ref{pleasca}) reads on $\ell^1(\G;\A)$
\begin{equation*}
\begin{aligned}
\big[{\sf ind}^\pi_\o (f)v\big](x)&=\sum_{z\in\G}\pi\big\{\alpha_{x^{-1}}[f(z)]\big\}\pi\big[\o(x^{-1},z)\big]v(z^{-1}x)\\
&=\sum_{y\in\G}\pi\big\{\alpha_{x^{-1}}[f(xy^{-1})]\big\}\pi\big[\o(x^{-1},xy^{-1})\big]v(y)\,.
\end{aligned}
\end{equation*}
Comparing this with (\ref{ghana}) one concludes that
\begin{equation}
{\sf Int}^\pi_\o\circ\Gamma={\sf ind}^\pi_\o\,.
\end{equation}

\begin{Remark}\label{togo}
{\rm If $z\in\G$\,, then $\pi\circ\alpha_{z^{-1}}=:\pi_z$ is a new $^*$-representation of $\A$\,. Thus we can construct all the representations $\,r^{\pi_z},L^{\pi_z}_\o,{\sf ind}^{\pi_z}_\o$ and ${\sf Int}_\o^{\pi_z}$\,. Defining
$$
R:\ell^2(\G;\H)\to \ell^2(\G;\H)\,,\quad (R v)(x):=\pi[\o(z^{-1},x^{-1})]v(xz)
$$
one gets easily the relations $\,r^{\pi_z}\!\overset{R}{\sim} r^{\pi}$\,, $L^{\pi_z}\!\overset{R}{\sim} L^{\pi}$ and (consequently) ${\sf ind}^{\pi_z}\!\overset{R}{\sim} {\sf ind}^{\pi}$\,, in which the notation means unitary equivalence. For instance:
$$
\begin{aligned}
\big[R\,r^\pi(\varphi)v\big](x)&=\pi[\o(z^{-1}\!,x^{-1})]\pi\big[\alpha_{(xz)^{-1}}(\varphi)]\big]v(xz)\\
&=\pi\big\{\alpha_{z^{-1}}\!\big[\alpha_{x^{-1}}(\varphi)]\big\}\pi[\o(z^{-1}\!,x^{-1})]v(xz)\\
&=\big[r^{\pi_z}(\varphi)R v\big](x)
\end{aligned}
$$
and
$$
\begin{aligned}
\big[R\,L^\pi(y) v\big](x)&=\pi[\o(z^{-1},x^{-1})]\big[L^\pi(y) v\big](xz)\\
&=\pi\big[\o(z^{-1},x^{-1})\big]\pi[\o(z^{-1}x^{-1},y)]v(y^{-1}xz)\\
&=\pi\big\{\alpha_{z^{-1}}[\o(x^{-1}\!,y)]\big\}\pi[\o(z^{-1}\!,x^{-1}y)]v(y^{-1}xz)\\
&=\pi_z\big[\o(x^{-1}\!,y)\big](R v)(y^{-1}x)\\
&=\big[L^{\pi_z}\!(y) R v\big](x)\,.
\end{aligned}
$$
The unitary equivalence ${\sf Int}_\o^{\pi_z}\!\overset{R}{\sim} {\sf Int}_\o^{\pi}$ only holds when restricted to the subspace $\mathscr K_{\alpha,\o}^{\L}(\G\times\G;\A)_{\rm cov}$\,.
}
\end{Remark}

\subsection{Twisted actions on Abelian $C^*$-algebras}\label{fuurt}

We assume now that the $C^*$-algebra $\A$ (unital, for simplicity) is Abelian. By Gelfand theory, it is enough to take it of the form $\,C(\Si):=\{\varphi:\Si\to\mathbb C\,\mid\,\varphi\ {\rm is \ continuous}\}$ for some Hausdorff compact topological space $\Si$ (homeomorphic to the Gelfand spectrum of $\A$)\,. Then the action $\alpha$ is derived from a continuous action of $\G$ by homeomorphisms $\{\si\mapsto x\cdot\si\mid x\in\G\}$ of $\Si$ by 
$$
\big[\alpha_x(\varphi)\big](\si):=\varphi\big(x^{-1}\!\cdot\si\big)\,.
$$ 
If $f\in \ell^1\big(\G;C(\Si)\big)$\,, $x\in\G$ and $\si\in\Si$ we are going to use the notation $f(x;\si)$ instead of $[f(x)](\si)$\,. In the same vein, the $C(\Si)$-valued kernels and the $2$-cocycle will be regarded as functions of three variables; for instance
$$
\o:\G\times\G\times\Si\to\mathbb T:=\{\zeta\in\mathbb C\,\mid\,|\zeta|=1\}\,,\quad\o(x,y;\si):=[\o(x,y)](\si)\,.
$$
The algebraic structure on $\L\big(\G;C(\Si)\big)$ becomes
\begin{equation}\label{shetland}
(f\diamond_{\alpha,\o}\!g)(x;\si):=\sum_{y\in\G}f(y;\si)g(y^{-1}x;y^{-1}\!\cdot\si)\,\o(y,y^{-1}x;\si)\,,
\end{equation}
\begin{equation}\label{noilehebride}
f^{\diamond_{\alpha,\o}}(x;\si):=\overline{\o(x,x^{-1};\si)}\,\overline{f(x^{-1};x^{-1}\!\cdot\si)}
\end{equation}
and that on $\K_{\alpha,\o}^\L\big(\G\times\G;C(\Si)\big)$ reads
\begin{equation}\label{suediaca}
(K\bullet_{\alpha,\o}\!L)(x,y;\si):=\sum_{z\in\G}K(x,z;\si)L(z,y;\si)\,\o(xz^{-1},zy^{-1};x^{-1}\!\cdot\si)\,,
\end{equation}
\begin{equation}\label{norvegiaca}
K^{\bullet_{\alpha,\o}}(x,y;\si):=\overline{\o(xy^{-1},yx^{-1};x^{-1}\!\cdot\si)}\,\overline{K(y,x;\si)}\,.
\end{equation}

We are going to indicate now two types of Hilbert space representations for twisted crossed products or for algebras of kernels, using Remark \ref{ulei} as a starting point.

\medskip
{\bf A.} Let us fix  some point $\si_0$ of $\Si$ with orbit $\mathcal O^{\si_0}:=\alpha_\G(\si_0)$ and quasi-orbit $\mathcal Q^{\si_0}:=\overline{\mathcal O^{\si_0}}$\,. The map $\alpha^{\si_0}:\G\rightarrow\Si$ given by $\alpha^{\si_0}(x):=\alpha_x(\si_0)$ is continuous and its range coincides with $\mathcal O^{\si_0}$, so this range is dense in $\mathcal Q^{\si_0}$. Taking $\H:=\mathbb C$\,, we set
\begin{equation}\label{swaziland}
\pi\equiv\delta_{\si_0}:C(\Si)\to\mathbb B(\mathbb C)=\mathbb C\,,\quad\delta_{\si_0}(\varphi):=\varphi(\si_0)\,.
\end{equation}
This leads as in Remark \ref{ulei} to the covariant representation $\big(r^{\si_0},L^{\si_0}_\o,\ell^2(\G;\mathbb C)\equiv\ell^2(\G)\big)$\,, where
$r^{\si_0}(\varphi)$ is the operator of multiplication by the function $\varphi\circ\alpha^{\si_0}$ for every $\varphi\in C(\Si)$ and 
$$
\big[L^{\si_0}_\o(y)v\big](x)=\o(x^{-1}\!,y;\si_0)v(y^{-1}x)\,.
$$
Thus, for every admissible algebra $\L(\G)$\,, one gets $^*$-representations
\begin{equation*}
{\sf ind}^{\si_0}_\o:\L\big(\G;C(\Si)\big)\to\mathbb B\big[\ell^2(\G)\big]\,,
\end{equation*}
\begin{equation}\label{vampirel}
\big[{\sf ind}^{\si_0}_\o\!(f)v\big](x)=\sum_{y\in\G}f\big(xy^{-1};x\cdot\si_0\big)\,\o\big(x^{-1}\!,xy^{-1};\si_0\big)v(y)
\end{equation}
and
\begin{equation*}
{\sf Int}^{\si_0}_\o:\mathscr K^\L_{\alpha,\o}\big(\G\times\G;C(\Si))\to\mathbb B\big[\ell^2(\G)\big]\,,
\end{equation*}
\begin{equation}\label{vampiras}
\big[{\sf Int}^{\si_0}_\o\!(K)v\big](x)=\sum_{y\in\G}K(x,y;\si_0)\,\o\big(x^{-1}\!,xy^{-1};\si_0\big)v(y)\,.
\end{equation}

\begin{Remark}\label{sahara}
{\rm The $^*$-representation ${\sf ind}^{\si_0}_\o$ above only depends on the orbit, up to unitary equivalence. If $\si_0,\si_1$ are on the same orbit and $z$ is an element of the group for which $\alpha_z(\si_0)=\si_1$, then the unitary operator
$$
R:\ell^2(\G)\to \ell^2(\G)\,,\quad (Rv)(x):=\o(z^{-1}\!,x^{-1};\si_0)v(xz)
$$
implements the equivalence: one has $\,R\,{\sf ind}_\o^{\si_0}\!(f)={\sf ind}_\o^{\si_1}\!(f)R\,$ for every $f\in \L\big(\G;C(\Si)\big)$\,. It is also true that $\,R\,{\sf Int}^{\si_0}_\o\!(K)={\sf Int}^{\si_1}_\o\!(K)R\,$ if $K$ is covariant, but on general kernels the connection fails. All these statements follow from Remark \ref{togo}.
}
\end{Remark}

{\bf B.} Assume now that $\mu$ is a Borel measure on $\Si$\,, invariant under the action $\alpha$\,. Then one has a $^*$-representation $\xi$ of $\A$ in the Hilbert space $\mathcal K:=L^2(\Si;\mu)$ given by 
\begin{equation}\label{bohr}
[\xi(\varphi)w_0](\si):=\varphi(\si)w_0(\si)\,.
\end{equation}
It admits an amplification $r^\xi$ representing $C(\Si)$ in $\ell^2(\G;\mathcal K)\cong \ell^2(\G)\otimes L^2(\Si)\cong L^2(\G\times\Si)$ by
$$
[r^\xi(\varphi)w](x;\si):=\big(\xi[\alpha_{x^{-1}}(\varphi)]w(x)\big)(\si)=\varphi(x\cdot\si)\,w(x;\si)\,.
$$
One also sets
$$
L^\xi_\o:\G\to\mathcal U\big[L^2(\G\times\Si)\big]\,,\quad \big[L^\xi_\o(y) w\big](x;\si):=\o\big(x^{-1}\!,y;\si)w\big(y^{-1}x;\si\big)\,.
$$
Clearly, one has the direct integral decomposition
$$
{\sf ind}^\xi_\o:=r^\xi\rtimes L^\xi_\o=\int_\Si^\otimes\!{\sf ind}_\o^\si\,\d\mu(\si)\,\quad{\sf Int}^\xi_\o=\int_\Si^\otimes\!{\sf Int}_\o^\si\,\d\mu(\si)
$$
where, for example, the $^*$-representation ${\sf Int}^\xi_\o:\mathscr K^\L_{\alpha,\o}\big(\G\times\G;C(\Si))\to\mathbb B\big[L^2(\G\times\Si)\big]$ is given by
$$
\big(\big[{\sf Int}^\xi_\o (K)\big]w\big)(x;\si)=\sum_{y\in\G}K(x,y;\si)\,\o\big(x^{-1}\!,xy^{-1};\si\big)w(y;\si)\,.
$$

\subsection{The standard case}\label{firtoinog}

For elements $\varphi$ of the $C^*$-algebra $\ell^\infty(\G)$ one sets 
\begin{equation}\label{austria}
\left[\alpha_x(\varphi)\right]\!(y):=\varphi(x^{-1}y)\,,\quad\ \forall\,x,y\in\G\,.
\end{equation}
We consider a unital $C^*$-subalgebra $\A(\G)$ of $\ell^\infty(\G)$ which is invariant under the action $\alpha$\,. If the $\alpha$-$2$-cocycle $\o$ satisfies $\omega(x,y;\cdot)\in\A(\G)$ for any elements $x,y$\,, then by restriction we form the twisted $C^*$-dynamical system $(\A(\G),\alpha,\o)$ and all the framework above is available. 

One can define on $\mathscr K^{\L}_{\alpha,\o}(\G\times\G;\A(\G))$ the transformation $\Upsilon$ given by
\begin{equation}\label{aricioara}
{\sf K}(x,y)\equiv(\Upsilon K)(x,y):=K(x,y;{\sf e})\,.
\end{equation}
\begin{Definition}\label{pui}
The range of the map $\Upsilon$ will be denoted by $\mathfrak K^{\L,\A}_{\,\o}(\G\times\G)$\,;  it is composed of all the two-variable scalar-valued kernels ${\sf K}\equiv\Upsilon K:\G\times\G\rightarrow\mathbb C$ such that for some $k\in\L(\G)$
\begin{equation}
|{\sf K}(x,y)|\le|k(xy^{-1})|\,,\quad\forall\,x,y\in\G
\end{equation}
and the function $z\mapsto{\sf K}(xz,yz)$ belongs to $\A(\G)$ for all $x,y\in\G$\,. 
\end{Definition}

If $\L(\G)=\ell^1(\G)$\,, one uses the simpified notation $\mathfrak K^{\A}_{\,\o}(\G\times\G)$\,. 
Note that the condition imposed on $z\mapsto{\sf K}(xz,yz)$ in the definition of $\mathfrak K^{\L,\A}_{\,\o}(\G\times\G)$ becomes vacuous if $\A(\G)=\ell^\infty(\G)$\,.

\medskip
The linear surjection $\Upsilon$ is highly non-injective. However, on $\K_{\alpha,\o}^\L\big(\G\times\G;\A(\G)\big)_{\!\rm cov}$\,, due to the covariance condition (\ref{lilican}), it becomes injective and, together with Proposition \ref{finlanda}, yields isometric isomorphisms of Banach $^*$-algebras
\begin{equation}\label{aripat}
\L_{\alpha,\o}\big(\G;\A(\G)\big)\overset{\Gamma}{\longrightarrow}\K_{\alpha,\o}^\L\big(\G\times\G;\A(\G)\big)_{\!\rm cov}\overset{\Upsilon}{\longrightarrow}\mathfrak K^{\L,\A}_{\,\o}(\G\times\G)\,.
\end{equation}
By composing one gets 
\begin{equation}\label{puiut}
[(\Upsilon\circ\Gamma)f](x,y)=f(xy^{-1};x)
\end{equation} 
for every $f\in\L_{\alpha,\o}\big(\G;\A(\G)\big)$ (see also \cite[Prop. 3.6]{BB} for a related result).

\medskip
The relevant algebraic structure on $\mathfrak K^{\L,\A}_{\,\o}(\G\times\G)$ is
\begin{equation}\label{aripioara}
\big({\sf K}\bullet_{\o}\!{\sf L}\big)(x,y):=\sum_{z\in\G}{\sf K}(x,z){\sf L}(z,y)\,\o(xz^{-1},zy^{-1};x^{-1})\,,
\end{equation}
\begin{equation}\label{aripiaca}
{\sf K}^{\bullet_{\o}}(x,y):=\overline{\o(xy^{-1},yx^{-1};x^{-1})}\,\overline{{\sf K}(y,x)}
\end{equation}
and the norm reads, once again by the covariance condition,
\begin{equation}\label{aripia}
\p\!{\sf K}\!\p_{\mathfrak K^{\L,\A}}\,:=\inf\left\{\p\!k\!\p_{\L}\,\mid\, |{\sf K}(x,y)|\,\le\,|k(xy^{-1})|\,,\ \,\forall\,x,y\in\G\right\}.
\end{equation}

\begin{Remark}\label{iepuroi}
{\rm For $\,\o=1$\,, $\L(\G)=\ell^1(\G)$ and $\A(\G)=\ell^\infty(\G)$\,, one essentially gets {\it the Banach $^*$-algebra $CD(\G)$ of convolution-dominated matrices on the discrete group $\G$} as defined in \cite{FGL} (cf. also references therein).
}
\end{Remark}

\begin{Remark}\label{iepugas}
{\rm Along the lines of Section \ref{fuurt}, one defines $\Sigma$ to be the Gelfand spectrum of $\A(\G)$\,. As a nice simple example, suppose that $\A(\G)$ consists of all the almost periodic functions on the discrete group $\G$\,. Then $\Si$ is the Bohr compact group $\beta G$ associated to $\G$ and its Haar measure $\mu$ can be used to define the (generally non-separable) Besicovich-type space $L^2(\beta G;\mu)$\,, the representation (\ref{bohr}) and the induced representations it defines.}
\end{Remark}

\begin{Remark}\label{puy}
{\rm To make the connection with Section \ref{fuurt} stronger, we could assume that the space $c_0(\G)$ of all  the functions $\varphi:\G\mapsto\mathbb C$ converging to zero at infinity (an $\alpha$-invariant closed ideal of $\ell^\infty(\G)$) is contained in $\A(\G)$.
Then $\Sigma$ is homeomorphic (and will be identified) to a compactification of the discrete group $\G$ and $\alpha$ extends to an action of $\G$ by homeomorphisms of $\Si$\,. Thus $\G$ is a dense orbit in the compact dynamical system $(\Si,\alpha,\G)$ and, as mentioned in Remark \ref{sahara}, the covariant representations $(r^z,L_\o^z)$ defined by various points $z\in\G\subset\Si$ are unitarily equivalent. 
}
\end{Remark}

\medskip
As a direct consequence of (\ref{aripat}) and Corollary \ref{terminal} one gets

\begin{Corollary}\label{iepuras}
Let $\,\G$ be a rigidly $\L$-symmetric discrete group for some admissible algebra $\L(\G)$  and $\A(\G)$ a unital $C^*$-subalgebra of $\,\ell^\infty(\G)$ which is invariant under translations. Then the Banach $^*$-algebra $\mathfrak K^{\L,\A}_{\,\o}(\G\times\G)$ of twisted kernels is symmetric for every $2$-cocycle $\o:\G\times\G\rightarrow\A(\G)$\,.
\end{Corollary}

Some consequences and examples can be inferred from Section \ref{trow}. In particular, Corollary \ref{japita} can be rephrased in terms of the symmetry of $\mathfrak K^{\ell^{1,\nu}\!,\A}_{\,\o}(\G\times\G)$ for weights $\nu$ satisfying conditions (\ref{ugrs}) and (\ref{magaoaie}) \,.

\medskip
In this standard case, besides the two types of Hilbert space representations {\bf A} and {\bf B} introduced in Section \ref{fuurt}, there is an interesting third one that we now present. It will be an ingredient of the proof of Corollary \ref{papita}.

\medskip
{\bf C.} One starts with the representation 
\begin{equation}\label{ienuparas}
\rho:\A(\G)\rightarrow\mathbb B\big[\ell^2(\G)\big]\,,\quad [\rho(\varphi)a_0](y):=\varphi(y)a_0(y)\,.
\end{equation}
In spite of its similarity with (\ref{bohr}), the representations $\xi$ and $\rho$ are different; very often the Hilbert space $L^2(\Si;\mu)$ is non-separable even if $\G$ is countable, as indicated in Remark \ref{iepugas}. 

We apply to $\rho$ the inducing procedure described in Remark \ref{ulei} (see also \cite[pag. 496-497]{FGL} for another point of view in the convolution-dominated case). Keeping in mind the identification $\ell^2(\G)\otimes\ell^2(\G)\cong\ell^2(\G\times\G)$\,, the asociated covariant representation is given by
\begin{equation}\label{pap}
[r^\rho(\varphi)a](x;z)=\varphi(xz)a(x;z)\,,\quad\big[L_\o^\rho(y)a\big](x;z)=\o(x^{-1},y;z)a(y^{-1}x;z)\,.
\end{equation}
The integrated form 
\begin{equation}\label{papas}
{\sf ind}^\rho_\o\equiv r^\rho\!\rtimes\!L^\rho_\o:\ell^1_{\alpha,\o}(\G;\A(\G))\rightarrow\mathbb B\!\left[\ell^2(\G\times\G)\right]
\end{equation} 
is a regular $^*$-representation. In terms of the representations associated to the points of the group $z\in\G$ and using the isomorphism $\ell^2(\G\times\G)\cong\bigoplus_{z\in\G}\ell^2(\G)$\,, one has the direct sum decompositions
\begin{equation}\label{papastratos}
{\sf ind}^\rho_\o\cong\bigoplus_{z\in\G}{\sf ind}^z_\o\,,\quad{\sf Int}^\rho_\o\cong\bigoplus_{z\in\G}{\sf Int}^z_\o\,.
\end{equation}

We return now to formula (\ref{vampiras}). In particular, taking $\si_0={\sf e}$ and setting $\lambda(x,y):=\o\big(x^{-1}\!,xy^{-1};{\sf e}\big)$ we rephrase it as
\begin{equation}\label{segesange}
\big[{\sf Int}_\lambda\!({\sf K})v\big](x):=\big[{\sf Int}^{\sf e}_\o\!({\sf K})v\big](x)=\sum_{y\in\G}{\sf K}(x,y)\lambda(x,y)v(y)\,,
\end{equation} 
which is seen as a $^*$-representation of the Banach $^*$-algebra $\mathfrak K^{\A}_{\,\o}(\G\times\G)$ by "twisted integral operators" in the Hilbert space $\ell^2(\G)$\,. Its range will be denoted by
\begin{equation}\label{puyut}
\mathbb I^{\A}_\lambda(\G):={\sf Int}_\lambda\!\left[\mathfrak K^{\A}_{\o}(\G\times\G)\right]={\sf ind}^{\sf e}_\o[\ell^1_{\alpha,\omega}(\G;\A(\G))]\,.
\end{equation} 

\begin{Corollary}\label{papita}
Let the group $\G$ be discrete, amenable and rigidly symmetric and suppose that the $2$-cocycle $\o$ is $\A(\G)$-valued for some unital $C^*$-subalgebra $\A(\G)$ of $\,\ell^\infty(\G)$ that is invariant under translations. Then $\,\mathbb I^{\A}_\lambda(\G)$  forms a $^*$-subalgebra of $\,\mathbb B\!\left[\ell^2(\G)\right]$ that is inverse closed and Frehdolm inverse closed. 
\end{Corollary}

\begin{proof}
Any $^*$-representation of a Banach $^*$-algebra extends to a $^*$-representation of its enveloping $C^*$-algebra. In particular we get representations (also denoted by) ${\sf ind}^\rho_\o:\A(\G)\!\rtimes_\alpha^\o\!\G\rightarrow\mathbb B\big[\ell^2(\G\times\G)\big]$ and ${\sf ind}^z_\o:\A(\G)\!\rtimes_\alpha^\o\!\G\rightarrow\mathbb B\big[\ell^2(\G)\big]$\,; by (\ref{papastratos}) they satisfy ${\sf ind}^\rho_\o\cong\bigoplus_{z\in\G}{\sf ind}^z_\o$\,.

Recall that the representations $\big({\sf ind}^z_\o\big)_{z\in\G}$ are mutually unitarily equivalent (by Remark \ref{sahara}, for instance). Thus ${\sf ind}^\rho_\o$ and ${\sf ind}^z_\o$ are simultaneously faithful; this does happen if the group $\G$ is amenable \cite{PR1}. In particular both ${\sf ind}^{\sf e}_\o$ and ${\sf Int}_\lambda$ are faithful at the level of the respective enveloping $C^*$-algebras.

Then our result folows from Theorem \ref{monaco}, Corollary \ref{dorinel} and Example \ref{dorinet}.
\end{proof}

\bigskip
{\bf Acknowledgements:} The author has been supported by the Chilean Science Foundation {\it Fondecyt} under the Grant 1120300 and by N\'ucleo Milenio de F\'isica Matem\'atica RC120002. He is grateful to Professor D. Poguntke for useful criticism.


\bigskip
\bigskip
\bigskip
{\bf Address:}\\
Departamento de Matem\'aticas\\ 
Universidad de Chile\\
Las Palmeras 3425, Casilla 653\\ 
Santiago, Chile\\
\emph{E-mail:} Marius.Mantoiu@imar.ro

\end{document}